\documentclass{amsart}
\usepackage{amsmath}
\usepackage{amsfonts}
\usepackage{amssymb}
\usepackage{graphicx}

\input{tcilatex}

\begin{document}
\title[Nuclear and type I crossed products]{Nuclear and type I crossed
products of C*-algebras by group and compact quantum group actions}
\author{Raluca Dumitru and Costel Peligrad}
\address{Raluca Dumitru: Department of Mathematics and Statistics,
University of North Florida, 1 UNF Drive, Jacksonville, Florida 32224;
Institute of Mathematics of the Romanian Academy, Bucharest, Romania; E-mail
address: raluca.dumitru@unf.edu}
\address{Costel Peligrad: Department of Mathematical Sciences, University of
Cincinnati, 610A Old Chemistry Building, Cincinnati, OH 45221; E-mail
address: costel.peligrad@uc.edu}
\maketitle

\begin{abstract}
If $A$ is a C*-algebra, $G$ a locally compact group, $K\subset G$ a compact
subgroup and $\alpha:G\rightarrow Aut(A)$ a continuous homomorphism, let $%
A\times_{\alpha}G$ denote the crossed product. In this paper we prove that $%
A\times_{\alpha}G$ is nuclear (respectively type I or liminal) if and only
if certain hereditary C*-subalgebras, $S_{\pi}$, $\mathcal{I}_{\pi}\subset
A\times_{\alpha}G$ $\pi\in\widehat{K}$, are nuclear (respectively type I or
liminal). These algebras are the analogs of the algebras of spherical
functions considered by R. Godement for groups with large compact subgroups.
If $K=G$ is a compact group or a compact quantum group, the algebras $%
S_{\pi} $ are stably isomorphic with the fixed point algebras $A\otimes
B(H_{\pi })^{\alpha\otimes ad\pi}$ where $H_{\pi}$ is the Hilbert space of
the representation $\pi.$
\end{abstract}

\section{Introduction and preliminary results}

Let $G$ be a locally compact group and $K\subset G$ a compact subgroup. In 
\cite{godement} (see also \cite{warner}) the study of $\widehat{G}$, the set
of equivalence classes of irreducible representations of $G$ is reduced to
the study of $\widehat{K}$ and the representations of certain classes of
spherical functions. In this paper we extend this approach to the case of
crossed products of C*-algebras by locally compact group and compact quantum
group actions. Let $(A,G,\alpha)$ be a C*-dynamical system and let $K\subset
G$ be a compact subgroup.

In \cite{peligrad} we defined the C*-algebras $S_{\pi}$, $\mathcal{I}%
_{\pi}\subset A\times_{\alpha}G$, $\pi\in\widehat{K}$ where $\widehat{K}$ is
the set of all equivalence classes of unitary representations of $K.$ These
are the analogs of the algebras of the algebras of spherical functions. For
the case $K=G$, these algebras were previously defined by Landstad in \cite%
{landstad}.

Recently, in \cite{raljfa,ralpelspectra}, we have extended the study of
these algebras to the case of compact quantum group actions on C*-algebras.
If $K=G$ is a compact group or a compact quantum group, the algebras $%
S_{\pi} $ are stably isomorphic with the fixed point algebras $A\otimes
B(H_{\pi})^{\alpha\otimes ad\pi}$ where $H_{\pi}$ is the Hilbert space of
the representation $\pi$. In this section we will review some definitions
and preliminary results.

\subsection{Preliminaries on actions of compact groups on C*-algebras.}

\ 

Let $K$ be a compact group and denote by $\widehat{K}$ the set of all
equivalence classes of irreducible, unitary representations of $K$. Let $%
\delta :K\rightarrow Aut(A)$ be an action of $K$ on a C*-algebra $A.$ Let $%
\pi \in \widehat{K\text{.}}$ If $\pi _{ij}(g)$ are the coefficients of $\pi
_{g}$ in a fixed basis of the Hilbert space $H_{\pi }$ of the representation 
$\pi ,1\leq i,j\leq d_{\pi }$ we define the character of $\pi ,\chi _{\pi
}(g)=d_{\pi }tr(\pi _{g^{-1}})=d_{\pi }\sum \overline{\pi _{ii}(g)},g\in K$
where $d_{\pi }$ is the dimension of the representation $\pi .$ We consider
the following mapping from $B$ into itself :

\begin{center}
$P^{\pi ,\delta }(a)=\int_{K}\chi _{\pi }(k)\delta _{k}(a)dk$
\end{center}

We define the spectral subspaces of the action $\delta$

\begin{center}
$A_{1}^{\delta }(\pi )=\left\{ a\in A|P^{\pi ,\delta }(a)=a\right\} $, $\pi
\in \widehat{K}$
\end{center}

In particular if $\pi=\pi_{0},$ is the trivial one dimensional
representation, $A_{1}^{\delta}(\pi_{0})=A^{\delta}$ is the algebra of fixed
elements under the action $\delta$. In this case, the projection $%
P^{\pi_{0},\delta}$ of $A$ onto $A^{\delta}$ is a completely positive map.
Indeed, the extension of $P^{\pi_{0},\delta}$ to $M_{n}(A)$ is the
projection of this latter C*-algebra onto its fixed point algebra with
respect to the action $\alpha\otimes id$ where $id$ is the trivial action of 
$G$ on $M_{n}=B(H_{n})$ where $H_{n}$ is the Hilbert space of dimension $n$.

\subsection{Algebras of spherical functions inside the crossed product}

\ 

Let now $(A,G,\alpha)$ be a C*-dynamical system with $G$ a locally compact
group and $K\subset G$ a compact subgroup. Denote by $A\times_{\alpha}G$ the
corresponding crossed product (see for instance \cite{pedersen}). Then the
algebra $C(K)$ of all continuous functions on $G$ can be embedded as follows
in the multiplier algebra $M(A\times_{\alpha}G)$ of $A\times_{\alpha}G$: If $%
\varphi\in C(K)$ and $y\in C_{c}(G,A),$the dense subalgebra of $A\times
_{\alpha}G$ consisting of continuous functions with compact support from $G$
to $A$, then

\begin{center}
$(\varphi y)(g)=\int_{K}\varphi(k)\alpha_{k}(y(k^{-1}g))dk$
\end{center}

and

\begin{center}
$(y\varphi)(g)=\int_{K}\varphi(k)y(gk)dk$
\end{center}

In particular, if $\varphi=\chi_{\pi},$ $\varphi$ is a projection in $%
M(A\times_{\alpha}G)$ and if $\pi_{1}$ and $\pi_{2}$ are distinct elements
in $\widehat{K\text{,}}$ the projections $\chi_{\pi_{1}}$ and $%
\chi_{\pi_{2}} $ are orthogonal. We need the following results from [\cite%
{peligrad}, Lemma 2.5.]:

\begin{remark}
\label{Lemma2.5JFA}The following statements hold:\newline i) If $\pi_{1} \neq\pi_{2}$ in $\widehat{K}$ then the projections $\chi_{\pi_{1}}$ and
$\chi_{\pi_{2}}$ are orthogonal in $M(A\times_{\alpha}G)$.\newline ii) $\sum_{\pi}\chi_{\pi}=I$, where $I$ is the identity of the bidual $(A\times_{\alpha}G)^{\star\star}$ of $A\times_{\alpha}G$.
\end{remark}

If $\pi\in\widehat{K}$, denote $S_{\pi}=\overline{\chi_{\pi}(A\times_{%
\alpha}G)\chi_{\pi}}$, where the closure is taken in the norm topology of $%
A\times_{\alpha}G$ Then, it is immediate that $S_{\pi}$ is strongly Morita
equivalent with the two sided ideal $J_{\pi}\overline{=(A\times_{\alpha}G)%
\chi_{\pi}(A\times_{\alpha}G)}$. Indeed, it can be easily verified that $X=%
\overline{(A\times_{\alpha}G)\chi_{\pi}}$ is an $S_{\pi}-J_{\pi}$
imprimitivity bimodule. We will consider next the action, $\delta$ of $K$ on 
$A\times_{\alpha}G$ defined as follows: If $y\in C_{c}(G,A)$ set $%
\delta_{k}(y)=\alpha_{k}(y(k^{-1}gk)$. Then $\delta_{k}$ extend to
automorphisms of $A\times_{\alpha}G$ and thus $\delta$ is an action of $K$
on $A\times_{\alpha}G$. The fixed point algebra $\mathcal{I=}%
(A\times_{\alpha}G)^{\delta}$ is called in \cite{peligrad} the algebra of
K-central elements of the crossed product $A\times_{\alpha}G$. Denote:

\begin{center}
$\mathcal{I}_{\pi}=\mathcal{I\cap}S_{\pi}$
\end{center}

Then, [\cite{peligrad}, Proposition 2.7.], we have

\begin{remark}
\label{JFA2.7}$S_{\pi}$ is $\ast-$isomorphic with $\mathcal{I}_{\pi}\otimes B(H_{\pi})$.
\end{remark}

If $G=K$ is a compact group, then by [\cite{landstad}, Lemma 3] we have:

\begin{remark}
\label{landstad lemma3}For every $\pi\in\widehat{G}$, $\mathcal{I}_{\pi}$ is $\ast-$ isomorphic with $(A\otimes B(H_{\pi}))^{\alpha\otimes ad\pi}$.
\end{remark}

\subsection{Compact quantum group actions on C*-algebras}

\ 

Let $\mathcal{G}=(B,\Delta)$ be a compact quantum group (\cite{wor1,wor2}).
Here, $B$ is a unital C*-algebra (which is the analog of the C*-algebra of
continuous functions in the group case) and $\Delta:B\rightarrow
B\otimes_{\min}B$ a $\ast$-homomorphism such that:

i) $(\Delta\otimes\iota)\Delta=(\iota\otimes\Delta)\Delta$, where $%
\iota:B\rightarrow B$ is the identity map and

ii) $\overline{\Delta(B)(1\otimes B)}=\overline{\Delta(B)(B\otimes1)}%
=B\otimes_{\min}B$.

Let $\widehat{\mathcal{G}}$ denote the set of all equivalence classes of
unitary representations of $\mathcal{G}$ or equivalently, the set of all
equivalence classes of irreducible unitary co-representations of $B$. For
each $\pi\in\widehat{\mathcal{G}}$, $\pi=\left[ \pi_{ij}\right] $, $%
\pi_{ij}\in B$ $1\leq i,j\leq d_{\pi}$,where $d_{\pi}$ is the dimension of $%
\pi$, let $\chi_{\pi}=\sum_{i}\pi_{ii}$ be the character of $\pi$ and let $%
F_{\pi}\in B(H_{\pi})$ be the positive, invertible matrix that intertwines $%
\pi$ with its double contragredient representation and such that $%
tr(F_{\pi})=tr(F_{\pi }^{-1})=M_{\pi}$. Then, with the notations in \cite%
{wor1}, $F_{\pi}=\left[ f_{1}(\pi_{ij})\right] $ where $f_{1}$ is a linear
functional on the $\ast -$subalgebra $\mathcal{B\subset}B$ that is linearly
spanned by $\left\{\pi_{ij}|\pi\in\widehat{\mathcal{G}},1\leq i,j\leq
d_{\pi}\right\}$. If $a\in B$ (respectively $\mathcal{B}$) and $\xi$ is a
linear functional on $B$ (respectively $\mathcal{B}$) we denote (\cite%
{wor1,wor2})

\begin{center}
$a\ast\xi=(\xi\otimes\iota)(\Delta(a))\in B$
\end{center}

Denote also by $\xi\cdot a$ the following linear functional on $B$
(respectively $\mathcal{B}$):

\begin{center}
$(\xi\cdot a)(b)=\xi(ab)$
\end{center}

If $h$ is the Haar state on $B$ let $h_{\pi}=M_{\pi}h\cdot(\chi_{\pi}\ast
f_{1})$. If $v_{r}$ is the right regular representation of $\mathcal{G}$,
the Fourier transform of $a\in B$ is defined as follows:

\begin{center}
$\widehat{a}=\mathcal{F}_{v_{r}}(a)=(\iota\otimes h\cdot a)(v_{r}^{\star})$
\end{center}

where $\mathcal{F}_{v_{r}}$ is the Fourier transform as defined by
Woronowicz in \cite{wor2}. Then the norm closure of the set $\widehat{B}%
=\left\{ \widehat {a}|a\in B\right\} $ is a C*-algebra called the dual of $B$
(\cite{baaj,wor2}) and $\widehat{B}$ is a subalgebra of the algebra of
compact operators, $\mathcal{C(}H_{h})$ on the Hilbert space $H_{h}$ of the
GNS representation of $B$ associated with the Haar state $h$.

Let $A$ be a C*-algebra and $\delta:A\rightarrow M(A\otimes B)$ be a $\ast-$
homomorphism of $A$ into the multiplier algebra of the minimal tensor
product $A\otimes B$. Then $\delta$ is called an action of $\mathcal{G}$ on $%
A$ (or a coaction of $B$ on $A$) if the following two conditions hold:%
\newline
a) $(\iota\otimes\Delta)\delta=(\delta\otimes\iota)\delta$ and\newline
b) $\overline{\delta(A)(1\otimes B)}=A\otimes B$

Let $\pi\in\widehat{\mathcal{G}}$. Denote $P^{\pi,\delta}(a)=(\iota\otimes
h_{\pi})(\delta(a)),a\in A$. Then $P^{\pi,\delta}$ is a contractive linear
map from $A$ into itself. In particular, if $\pi=\pi_{0}$ is the trivial one
dimensional representation, then $P^{\pi_{0},\delta}=(\iota\otimes h)\delta$
is the completely positive projection of norm $1$ of $A$ onto the fixed
point C*-subalgebra $A^{\delta}$.

The crossed product $A\times_{\delta}\mathcal{G}$ is by definition, (\cite%
{baaj,boca}), the norm closure of the set $\left\{
(\pi_{u}\otimes\pi_{h})(\delta (a)(1\otimes\widehat{b})|a\in A,b\in
B\right\} $, where $\pi_{u}$\ is the universal representation of $A$\ and $%
\pi_{h}$\ is the GNS representation of $B$ associated with the Haar state $h$%
.

Let $\pi\in\widehat{\mathcal{G}}$. If we denote $p_{\pi}=(\iota\otimes
h_{\pi})(v_{r}^{\star})$, then $\left\{ p_{\pi}\right\} _{\pi\in\widehat {%
\mathcal{G}}}$ are mutually orthogonal projections in $\widehat{B}$ and
therefore in $A\times_{\delta}\mathcal{G}$ (\cite{boca,raljfa}). For $\pi\in 
\widehat{\mathcal{G}}$ denote $\mathcal{S}_{\pi}=\overline{%
p_{\pi}(A\times_{\delta}\mathcal{G)}p_{\pi}}$. In [\cite{raljfa}, Lemma 3.3]
it is shown that $ad(v_{r})$ is an action of $\mathcal{G}$ on the crossed
product $A\times_{\delta}\mathcal{G}$ and the fixed point algebra $\mathcal{%
I=(}A\times_{\delta}\mathcal{G)}^{ad(v_{r})}$ of this action plays the role
of the $K-$central elements in the case of groups. Let $\mathcal{I}_{\pi}=%
\mathcal{I\cap S}_{\pi}$. Let $\delta_{\pi}$ be the following action of $%
\mathcal{G}$ on $A\otimes B(H_{\pi})$:

\begin{center}
$\delta_{\pi}(a\otimes m)=(\pi)_{23}(\delta(a))_{13}(1\otimes m\otimes
1)(\pi^{\ast})_{23}$
\end{center}

where the leg-numbering notation is the usual one (\cite{baaj,wor2}). The
above $\delta_{\pi}$ equals $\delta\otimes ad(\pi)$ in the case of compact
groups. Then, we have:

\begin{remark}
\label{ralucaanalogsof2.2,2.8andlandstadlemma}The following statements hold
true:\newline i) The projections $\left\{  p_{\pi}\right\}  _{\pi\in
\widehat{\mathcal{G}}}$ are mutually orthogonal and $\sum_{\pi}p_{\pi}=1$ in
the bidual $(A\times_{\delta}\mathcal{G)}^{\star\star}$\newline ii)
$\mathcal{S}_{\pi}$ is $\star-$isomorphic with $\mathcal{I}_{\pi}\otimes
B(H_{\pi})$\newline iii) $\mathcal{I}_{\pi}$ is $\star-$isomorphic with
$A\otimes B(H_{\pi})^{\delta_{\pi}}$
\end{remark}

\begin{proof}
Part i) is [\cite{raljfa}, Section 2.1., Equation (2) and the discussion
after that equation]. Part ii) is [\cite{raljfa}, Remark 3.5.] and Part iii)
is [\cite{raljfa}, Proposition 4.8.].
\end{proof}

\section{Nuclear and type I crossed products}

In this section we will state and prove our main results. We give necessary
and sufficient conditions for a crossed product to be nuclear or type I. Our
conditions are given in terms of the algebras of spherical functions inside
the crossed product and in case of compact groups or compact quantum groups,
in terms of the fixed point algebras of $A\otimes B(H_{\pi})$ for the
actions $\delta\otimes ad(\pi)$.

Recall that a C*-algebra $C$ is said to be of type I if for every factor
representation $T$ of $C$ the Von Neumann factor $T(C)^{\prime\prime}$ is a
type I factor. $C$ is called liminal if for every irreducible representation 
$T$ of $C$, $T(C)$ consists of compact operators.

A C*-algebra is called nuclear if its bidual, $C^{\ast\ast}$, is an
injective von Neumann algebra, i.e. if and only if there is a projection of
norm one from $B(H_{u})$ onto $C^{\ast\ast}$, where $H_{u}$ is the Hilbert
space of the universal representation of $C$. With the notations from
Section 1, we have the following :

\begin{remark}
\label{Cor 2.8JFA}Let $(A,G,\alpha)$ be a C*-dynamical system with G a locally compact group and let $K\subset G$ be a compact subgroup. The following three statements hold:\newline i) $S_{\pi}$ is nuclear if and only if $\mathcal{I}_{\pi}$ is nuclear\newline ii) $S_{\pi}$ is liminal if and only if $\mathcal{I}_{\pi}$ is liminal\newline iii) $S_{\pi}$ is type I if and only if $\mathcal{I}_{\pi}$ is type I
\end{remark}

\begin{proof}
These statements follow from Remark \ref{JFA2.7}.
\end{proof}

The following is the analog of the above Remark for the case of compact
quantum group actions:

\begin{remark}
\label{followsfromRalRemark3.5}Let $\mathcal{G}=(B,\Delta)$ be a compact quantum group and $\delta$ an action of $\mathcal{G}$ on a C*-algebra $A$. The
following conditions are equivalent:\newline i) $\mathcal{S}_{\pi}$ is nuclear if and inly if $(A\otimes B(H_{\pi}))^{\delta_{\pi}}$ is nuclear\newline ii)
$\mathcal{S}_{\pi}$ is liminal if and inly if $(A\otimes B(H_{\pi}))^{\delta_{\pi}}$ is liminal\newline iii) $\mathcal{S}_{\pi}$ is type I if and only if $(A\otimes B(H_{\pi}))^{\delta_{\pi}}$ is type I
\end{remark}

\begin{proof}
The result follows from Remark \ref{ralucaanalogsof2.2,2.8andlandstadlemma}.
\end{proof}

\subsection{Type I crossed products}

\ 

We start with the following general result:

\begin{lemma}
\label{typeIlemma}Let $C$ be a C*-algebra and $M(C)$ the multiplier algebra of $C.$ Let $\left\{  p_{\lambda}\right\}  \subset M(C)$ be a family of mutually
orthogonal projections of sum 1 in $C^{\ast\ast}$, the bidual of $C$. The following conditions are equivalent:\newline i) $C$ is type I (respectively
liminal) \newline ii) The hereditary subalgebras $S_{\lambda}=p_{\lambda}Cp_{\lambda}\subset C$ are type I (respectively liminal) for every $\lambda$.
\end{lemma}

\begin{proof}
Assume that $C$ is type I (respectively liminal). Then $S_{\lambda}$ are
type I (respectively liminal) as C*-subalgebras of a type I (liminal)
C*-algebra.

Assume now that all $S_{\lambda}$ are type I (liminal). Let $T$ be a
nondegenerate factor representation (respectively an irreducible
representation) of $C$. Since, by assumption, $\sum p_{\lambda}=1$ it
follows that $\sum p_{\lambda}C$ is norm dense in $C$. Therefore, there is a 
$\lambda$ such that the restriction of $T$ to $p_{\lambda}C,$ $%
T|_{p_{\lambda}C}\neq0$. Then $T|_{J_{\lambda}}\neq0$, where $J_{\lambda}=%
\overline{Cp_{\lambda}C}$ is the two sided ideal of $C$ generated by $%
p_{\lambda}$. Since $T$ is a factor representation of $C$ (respectively an
irreducible representation of $C$) and the bicommutant $T(J_{\lambda})^{^{%
\prime\prime}}$ is a nonzero weakly closed ideal of $T(C)^{\prime\prime}$ it
follows that $T(J_{\lambda})^{^{\prime\prime}}=T(C)^{\prime\prime}$.
Therefore $T$ has the same type with $T|_{J_{\lambda}}$. On the other hand,
it can be checked that $J_{\lambda}$ is strongly Morita equivalent with $%
S_{\lambda}$ in the sense of Rieffel, \cite{rieffel}, with imprimitivity
bimodule $Cp_{\lambda}$. Therefore, since $S_{\lambda}$ is assumed to be
type I (respectively liminal), it follows from the discussion in \cite%
{rieffel} (respectively \cite{fell}) that $J_{\lambda}$ is type I
(respectively liminal). It then follows that the representation $T$ is a
type I representation (respectively $T(C)$ consists of compact operators).
Since $T$ was arbitrary, we are done.
\end{proof}

We will state next some consequences of the above Lemma.

\begin{theorem}
\label{typeIlocallycompact}Let $(A,G,\alpha)$ be a C*-dynamical system with $G$ a locally compact group and let $K\subset G$ be a compact subgroup. Then
the following conditions are equivalent:\newline i) $A\times_{\alpha}G$ is type I (respectively liminal)\newline ii) The hereditary C*-subalgebras
$S_{\pi}\subset A\times_{\alpha}G$, $\pi\in\widehat{K}$ are type I (respectively liminal)\newline iii) The C*-subalgebras of $K-$central
elements, $\mathcal{I}_{\pi}\subset S_{\pi},\pi\in\widehat{K}$ are type I (respectively liminal).
\end{theorem}

\begin{proof}
The equivalence of the conditions i)-iii) follows from Remarks \ref%
{Lemma2.5JFA} and \ref{Cor 2.8JFA} and Lemma \ref{typeIlemma}.
\end{proof}

If $G=K$ is a compact group, then the conditions i)-iii) in the above
theorem are equivalent with:

iv) The fixed point algebra $A^{\alpha}$ is type I (respectively liminal) [%
\cite{gootman}, Theorem 3.2].

We will prove next an analogous result for compact quantum group actions. In
[\cite{boca}, Theorem19] it is shown that the crossed product of a
C*-algebra by an ergodic action of a compact quantum group is a direct sum
of full algebras of compact operators, hence a liminal C*-algebra. Since, in
the ergodic case, $\mathcal{S}_{\pi}$ are finite dimensional, the next
result is an extension of Boca's result to the case of general compact
quantum group actions.

For compact quantum groups we have the following result:

\begin{theorem}
\label{typeIquantum}Let $\mathcal{G=(}B,\Delta)$ be a compact quantum group and $\delta$ an action of $\mathcal{G}$ on a C*-algebra $A$. The following
conditions are equivalent:\newline i) $A\times_{\delta}\mathcal{G}$ is type I (respectively liminal)\newline ii) The hereditary C*-subalgebras
$\mathcal{S}_{\pi}\subset A\times_{\delta}\mathcal{G}$, $\pi\in\widehat{\mathcal{G}}$, are type I (respectively liminal)\newline iii) The
C*-subalgebras $\mathcal{I}_{\pi}\subset\mathcal{S}_{\pi},\pi\in \widehat{\mathcal{G}}$, are type I (respectively liminal).\newline iv) The
C*-algebras $A\otimes B(H_{\pi})^{\delta_{\pi}}$, $\pi\in\widehat{\mathcal{G}}$ are type I (respectively liminal).
\end{theorem}

\begin{proof}
The result follows from Remark \ref{ralucaanalogsof2.2,2.8andlandstadlemma}
and Lemma \ref{typeIlemma}.
\end{proof}

\subsection{Nuclear crossed products}

\ 

We start with the following lemma which is certainly known but we could not
find a reference for it:

\begin{lemma}
\label{prepLemma}A C*-algebra $C$ is nuclear if and only if for every state $\varphi$ of $C$, $T_{\varphi}(C)^{\prime\prime}$ is an injective von Neumann
algebra, where $T_{\varphi}$ is the GNS representation of $C$ associated with $\varphi$.
\end{lemma}

\begin{proof}
If $C$ is nuclear then $C^{\ast\ast}$ is an injective von Neumann algebra [%
\cite{effros}, Theorem 6.4.]. Therefore, so is $T_{\varphi}(C)^{\prime%
\prime} $ which is isomorphic with an algebra of the form $eC^{\ast\ast}$
for a certain projection, $e\in(C^{\ast\ast})^{\prime}$.

Conversely, if $T_{\varphi}(C)^{\prime\prime}$ is injective for every state $%
\varphi$, let $\left\{ \varphi_{\iota}\right\} $ be a maximal family of
states for which the corresponding cyclic representations $%
T_{\varphi_{\iota}}$ are disjoint. Then $T_{\varphi_{\iota}}$ and $T=\oplus
T_{\varphi_{\iota}}$ can be extended to normal representations $\overline{%
T_{\varphi_{\iota}}\ }$ and $\overline{T\ }$ of $C^{\ast\ast}$ with $%
\overline{T\ }$ a normal isomorphism. Therefore, $C^{\ast\ast}$ is
isomorphic with $\oplus T_{\varphi_{\iota}}(C)^{\prime\prime}$. Since all $%
T_{\varphi_{\iota}}(C)^{\prime\prime}$ are injective (by assumption), from [%
\cite{effros}, Proposition 3.1.], it follows that $C^{\ast\ast}$ is
injective and thus $C$ is nuclear.
\end{proof}

Throughout the rest of this section all algebras, groups and quantum groups
are assumed to be separable. The following Lemma is the analog of Lemma \ref%
{typeIlemma} for the case of nuclear crossed products.

\begin{lemma}
\label{nuclearlemma}Let $C$ be a separable C*-algebra and $\left\{q_{\lambda}\right\}  \subset M(C)$ be a family of mutually orthogonal projections such that $\sum_{\lambda}q_{\lambda}=1$ in $C^{\star\star}$. The following statements are equivalent\newline i) $C$ is nuclear\newline ii) The hereditary C*-subalgebras $S_{\lambda}$ are nuclear for all $\lambda$.
\end{lemma}

\begin{proof}
Assume first that $C$ is nuclear. Then, by [\cite{choi}, Corollary 3.3 (4)],
every hereditary subalgebra of $C$ is nuclear. Hence $S_{\lambda}$ is
nuclear for every $\lambda$.\newline
Assume now that ii) holds that is : all $S_{\lambda}$ are nuclear
C*-algebras. We will show that for every cyclic representation $T_{\varphi}$
of $C$, $T_{\varphi} (C)^{\prime\prime}$ is injective and the result will
follow from the previous lemma. Let $\varphi$ be a state of $C$. Then, by
reduction theory, $\overline{T_{\varphi}}(C^{\ast\ast})=T_{\varphi}(C)^{%
\prime\prime}$ is the direct integral of factors $\overline{T_{\psi}}%
(C^{\ast\ast})=T_{\psi }(C)^{\prime\prime}$ where $\psi$ are factor states
of $C$, $\overline {T_{\varphi}}(C^{\ast\ast})=\int\overline{T_{\psi}}%
(C^{\ast\ast})d\mu (\psi)$ where $\mu$ is the central measure associated
with the state $\varphi$ and the integral is taken over the state space of $%
C $ [\cite{sakai}, Theorem 3.5.2.]. Applying [\cite{connes}, Proposition
6.5.], it follows that $T_{\varphi}(C)^{\prime\prime }$ is injective if and
only if almost all of the factors $T_{\psi}(C)^{\prime\prime}$ are
injective. We have, therefore, reduced our problem to the following:

Assuming that all hereditary subalgebras $S_{\lambda}$ are nuclear, show
that for every cyclic factor representation $T$ of $C$ we have that $%
T(C)^{\prime\prime}$ is an injective von Neumann algebra.

Let $T$ be a non degenerate cyclic factor representation of $C$. Since $%
\sum_{\lambda}q_{\lambda}=1$ in $C^{\ast\ast}$ there is a $\lambda$ such
that the restriction $T|_{q_{\lambda}C}\neq0$. Hence the restriction of $T$
to the closed two sided ideal $J_{\lambda}=\overline{Cq_{\lambda}C}$ is non
zero. Since $T$ is a factor representation and $J_{\lambda}$ is a two sided
ideal it follows that $T(C)^{\prime\prime}=T(J_{\lambda})^{\prime\prime}$.
We show next that under our assumptions $T(J_{\lambda})^{\prime\prime}$ is
injective and thus $T(C)^{\prime\prime}$ is injective. We noticed above that 
$S_{\lambda}$ is strongly Morita equivalent with $J_{\lambda}$. Since $C$ is
separable, so are $S_{\lambda\text{ }}$ and $J_{\lambda}$. By [\cite{brown},
Theorem 1.2.] $S_{\lambda\text{ }}$ and $J_{\lambda}$ are stably isomorphic.
Since $S_{\lambda\text{ }}$ is nuclear it follows that $J_{\lambda}$ is
nuclear. By Lemma \ref{prepLemma} we have that $T(J_{\lambda})^{\prime%
\prime} $ is injective and the proof is complete.
\end{proof}

From the proof of the previous lemma it follows:

\begin{corollary}
A separable C*-algebra $C$ is nuclear if and only if for every factor state $\psi$ of $C$, $T_{\psi}(C)^{\prime\prime}$ is an injective von Neumann algebra, where $T_{\psi}$ is the GNS representation of $C$ associated with $\psi$.
\end{corollary}

We can now state our main results of this section.

\begin{theorem}
\label{nuclearlocallycompact}Let $(A,G,\alpha)$ be a C*-dynamical system with $G$ a locally compact group and let $K\subset G$ be a compact subgroup. Then
the following conditions are equivalent:\newline i) $A\times_{\alpha}G$ is a nuclear C*-algebra\newline ii) The hereditary C*-subalgebras $S_{\pi}\subset
A\times_{\alpha}G$, $\pi\in\widehat{K}$ are nuclear\newline iii) The C*-subalgebras of $K-$central elements, $\mathcal{I}_{\pi}\subset S_{\pi}
,\pi\in\widehat{K}$ are nuclear\newline Furthermore, any of the previous three equivalent conditions implies\newline iv) $A$ is nuclear\newline In addition,
if $G$ is amenable, i.e if the group C*-algebra $C^{\ast}(G)$ is nuclear the conditions i)-iv) are equivalent.
\end{theorem}

\begin{proof}
The equivalence of the conditions i)-iii) follows from Remarks \ref%
{Lemma2.5JFA} and \ref{Cor 2.8JFA} and Lemma \ref{nuclearlemma}. On the
other hand, if the crossed product, $A\times_{\alpha}G$ is nuclear, then,
applying [\cite{raeburn}, Theorem 4.6.], it follows that $%
A\times_{\alpha}G\times_{\widehat{\alpha}}\widehat{G}$ is nuclear, where $%
\widehat{\alpha}$ is the dual coaction. Since by biduality this latter
crossed product is isomorphic with $A\otimes\mathcal{C(H)}$ where $\mathcal{%
C(H)}$ is the C*-algebra of compact operators on a certain Hilbert space, $%
\mathcal{H}$, it follows that $A$ is a nuclear C*-algebra. Finally, if $G$
is amenable and $A$ is nuclear, then by [\cite{green}, Proposition 14], the
crossed product $A\times_{\alpha}G$ is nuclear and therefore in this case iv)%
$\implies$i).
\end{proof}

In the proof of the implication i)$\mathcal{\Longrightarrow }$iv) of the
above theorem we have used the fact that every locally compact group is
co-amenable and Raeburn's result. The next result is the analog of the
previous one for the case of compact quantum groups. A compact quantum
group, $\mathcal{G}=(B,\Delta )$, is automatically amenable since $\widehat{B%
}$ is a subalgebra of compact operators, but not co-amenable, in general,
since $B$ is not necessarily nuclear.

We will state next the corresponding result for compact quantum group
actions.

\begin{theorem}
Let $(A,\mathcal{G},\delta)$ be a quantum C*-dynamical system with $\mathcal{G=(}B,\Delta)$ a compact quantum group. The following three
conditions are equivalent:\newline i) $A\times_{\alpha}\mathcal{G}$ is nuclear\newline ii) The hereditary C*-subalgebras $S_{\pi}\subset
A\times_{\alpha}\mathcal{G}$ are nuclear\newline iii) The C*-algebras $(A\otimes B(H_{\pi}))^{\delta_{\pi}}$ are nuclear.\newline Furthermore, each
of the above condition is implied by \newline iv) $A$ is a nuclear C*-algebra.\newline In addition, if the quantum group $\mathcal{G}$ is
co-amenable, i.e. if $B$ is a nuclear C*-algebra, then the conditions i)-iv) are equivalent with the following:\newline v) $A^{\delta}$ is nuclear.
\end{theorem}

\begin{proof}
The equivalence of i)-iii) follows from Lemma \ref%
{ralucaanalogsof2.2,2.8andlandstadlemma} and Lemma \ref{nuclearlemma}. We
now prove that iv) implies iii). Let $\pi\in\widehat{\mathcal{G}}$. If $A$
is nuclear, then $A\otimes B(H_{\pi})$ is a nuclear C*-algebra. The
projection of $A\otimes B(H_{\pi})$ onto the fixed point algebra $(A\otimes
B(H_{\pi }))^{\delta_{\pi}}$ is obviously a completely positive map.
Therefore, by [\cite{choi}, Corollary 3.4. (4)] it follows that $(A\otimes
B(H_{\pi}))^{\delta_{\pi}}$ is nuclear. Assume now that $\mathcal{G}$ is
co-amenable. Therefore, $B$ is nuclear. Then, by applying [\cite{doplicher},
Corollary 7] it follows that $A^{\delta}$ is nuclear if and only if $A$ is
nuclear and thus v)$\iff$iv). Since $S_{\pi_{0}}$ is isomorphic with $%
A^{\delta}$, we have that iii)$\implies$iv) and the proof is completed.
\end{proof}

\end{document}